\documentclass[a4paper,11pt,reqno]{amsart}

\usepackage{amsmath}
\usepackage{amsthm}
\usepackage{amstext}
\usepackage{amsfonts}
\usepackage{graphicx}
\usepackage{dsfont}

\usepackage[usenames,dvipsnames]{pstricks}
\usepackage{epsfig}
\usepackage{pstricks-add}

\numberwithin{equation}{section}

\newcommand{\R}{\mathbb R}
\newcommand{\C}{\mathbb C}
\newcommand{\N}{\mathbb N}

\newcommand{\be}{\begin{equation}}
\newcommand{\ee}{\end{equation}}
\newcommand{\ba}{\begin{eqnarray}}
\newcommand{\ea}{\end{eqnarray}}

\newcommand{\asy}{\stackrel{\textrm{as}}{<}}

\newtheorem{theorem}{Theorem}[section]
\newtheorem{proposition}[theorem]{Proposition}
\newtheorem{remark}[theorem]{Remark}
\newtheorem{lemma}[theorem]{Lemma}
\newtheorem{corollary}[theorem]{Corollary}

\begin{document}

\title[Reachable states for the heat equation]{On the reachable states for the boundary control of the heat equation}

\author{Philippe Martin}
\address{Centre Automatique et Syst\`emes, MINES ParisTech, PSL Research University\\
60 boulevard Saint-Michel, 75272 Paris Cedex 06, France}
\email{philippe.martin@mines-paristech.fr}

\author{Lionel Rosier}
\address{Centre Automatique et Syst\`emes, MINES ParisTech, PSL Research University\\
60 boulevard Saint-Michel, 75272 Paris Cedex 06, France}
\email{lionel.rosier@mines-paristech.fr}

\author{Pierre Rouchon}
\address{Centre Automatique et Syst\`emes, MINES ParisTech, PSL Research University\\
60 boulevard Saint-Michel, 75272 Paris Cedex 06, France}
\email{pierre.rouchon@mines-paristech.fr}

\begin{abstract}                          
We are interested in the determination of the reachable states for the boundary control of the one-dimensional heat equation. 
We consider either one or two boundary controls. We show that reachable states associated with square integrable controls can be extended to analytic functions on
some square of $\C$, and conversely, that analytic functions defined on a certain disk can be reached by using boundary controls
that are Gevrey functions of order 2.   The method of proof combines the flatness approach with some new Borel interpolation theorem in some Gevrey class with
a specified value of the loss in the uniform estimates of the successive derivatives of the interpolating function. 
\end{abstract}


\maketitle
\section{Introduction}

The null controllability of the heat equation has been extensively studied since the seventies. After the pioneering work \cite{FR} in the one-dimensional case using
biorthogonal families, sharp results were obtained in the N-dimensional case by using elliptic Carleman estimates \cite{LR} or parabolic Carleman estimates \cite{FI}.
An exact  controllability to the trajectories was also derived, even for nonlinear systems \cite{FI}.

By contrast, the issue of the exact controllability of the heat equation (or of a more general semilinear parabolic equation) is not well understood. For the sake of simplicity, let us 
consider the following control system  
 \ba
\psi _t-\psi _{xx}= 0,&& x\in (0,1),\ t\in (0,T), \label{P1}\\
\psi (0,t)= h_0(t),&&  t\in (0,T), \label{P2}\\
\psi (1,t)= h_1(t),&&  t\in (0,T), \label{P3}\\
\psi (x,0) = \psi _0(x),&& x\in (0,1),\label{P4}
 \ea 
where $\psi _0\in L^2(0,1)$, and $h_0,h_1\in L^2(0,T)$.  

As \eqref{P1}-\eqref{P4} is null controllable, there is no loss of generality is assuming that $\psi _0\equiv 0$. A state $\psi _1$ is said to be {\em reachable (from $0$ in time $T$)} if we can find
two control inputs $h_0,h_1\in L^2(0,T)$ such that the solution $\psi $ of \eqref{P1}-\eqref{P4} satisfies
\be
\label{P7}
\psi (x,T)=\psi _1(x), \quad \forall x \in (0,1). 
\ee  
Let $ A \psi :=\psi ''$ with domain $D(A):=H^2(0,1)\cap H^1_0(0,1)\subset L^2(0,1)$, 
and let $e_n(x)=\sqrt{2}\sin (n\pi x)$ for  $n\ge 1$ and $x\in (0,1)$. As is well known,  $(e_n)_{n\ge 1}$ is an orthonormal basis in $L^2(0,1)$ constituted
of eigenfunctions of $A$. Decompose $\psi _1$ as 
\be
\psi _1(x)=\sum_{n=1}^\infty c_n e_n (x) =\sum_{n=1}^\infty c_n \sqrt{2} \sin (n\pi x) . 
\label{P10}
\ee  
Then, from \cite{FR}, if we have for some $\varepsilon >0$
\be
\label{P11}
\sum_{n=1}^\infty |c_n|n^{-1} e^{ (1+\varepsilon ) n\pi } <\infty , 
\ee
then $\psi _1$ is a reachable terminal state. Note that the condition \eqref{P11} implies that 
\begin{enumerate}
\item[(i)] the function $\psi _1$ is analytic in $D(0,1+\varepsilon ) :=\{ z\in \C ; \ |z| < 1+\varepsilon \} $, for the series in \eqref{P10}  
converges uniformly in $\overline{ D(0,r) }$ for all $r<1+\varepsilon$;
\item[(ii)] 
\be
\sum_{n=1}^\infty |c_n|^2n^{2k} <\infty,\quad \forall k\in \N
\label{P12}
\ee
(that is, $\psi _1\in \cap _{k\ge 0} D(A^k)$), and hence 
\be
\label{P13}
\psi _1^{(2n)} (0)=\psi _1^{(2n)} (1)=0,\quad \forall n\in \N .
\ee
\end{enumerate}
More recently, it was proved in \cite{EZ} that any state $\psi _1$ decomposed as in \eqref{P10} is reachable  if 
\be
\label{P14} 
\sum_{n\ge 1} |c_n|^2 n e^{ 2n\pi }  <\infty ,
\ee
which again implies \eqref{P12} and \eqref{P13}. 

It turns out that \eqref{P13} is a very conservative condition, which excludes most of the usual analytic functions. As a matter of fact, the only  polynomial function satisfying \eqref{P13} is the null one. 
On the other hand, the condition \eqref{P13} is not very natural, for there is no reason that $h_0(T)=h_1(T)=0$. We shall see  that the {\em only} condition required for $ \psi  _1$ 
to be reachable is the analyticity of $ \psi _1$ on a sufficiently large open set in $\C$. 

Notations: If $\Omega$ is an open set in $\C$, we denote by $H(\Omega )$  
the set of holomorphic (complex analytic) functions $f:\Omega \to\C$.

The following result gathers together some of the main results contained in this paper.

\begin{theorem}
\label{thmintro}
1. Let $z_0=\frac{1}{2}$. If $\psi _1\in H (D( z_0 ,R/2))$ with $R>R_0:=e^{ (2e)^{-1} }\sim 1.2$, then $\psi _1$ is reachable from $0$ in any time $T>0$. Conversely, any reachable state belongs to 
$H( \{ z=x+iy; \ | x-\frac{1}{2} |  + |y| < \frac{1}{2}  \} )$. \\
2. If $\psi _1\in H(D(0,R))$ with $R>R_0$ and $\psi _1$ is odd, then $\psi _1$ is reachable from 0 in any time $T>0$ with only one boundary control at $x=1$ (i.e. $h_0\equiv 0$).   
Conversely, any reachable state with only one boundary control at $x=1$ is odd and it belongs to $ H( \{ z=x+iy; \ |x|+|y|< 1 \} ) $. 
\end{theorem}  

Thus, for given $a\in \R$, the function $\psi _1(x) :=[(x-\frac{1}{2})^2+a^2]^{-1}$ is reachable if $|a|>R_0/2\sim 0.6$, and it is not reachable if $|a|<1/2$. 

Figure \ref{fig1} is concerned with the reachable states associated with two boundary controls at $x=0,1$:   any reachable state can be extended to the red square as a (complex) analytic function; conversely, the restriction to $[0,1]$ of any analytic function on a disc containing the blue one is a reachable state.

\begin{figure}
\begin{center}
\includegraphics[width=0.3\textwidth]{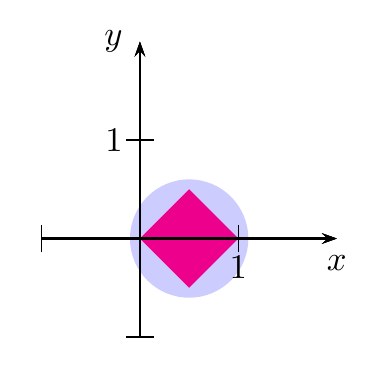} 
\end{center}
\caption{$\{ |x-\frac{1}{2} |+|y|<\frac{1}{2} \}$ (red) and  $D(z_0, \frac{R_0}{2} )$ (blue)}
\label{fig1}
\end{figure}

Figure \ref{fig2} is concerned with the reachable states associated with solely one boundary control at $x=1$:   any reachable state can be extended to the red square as an analytic (odd) function; conversely, the restriction to $[0,1]$ of any analytic (odd) function on a disc containing the blue one is a reachable state. 

\begin{figure}
\begin{center}
\includegraphics[width=0.3\textwidth]{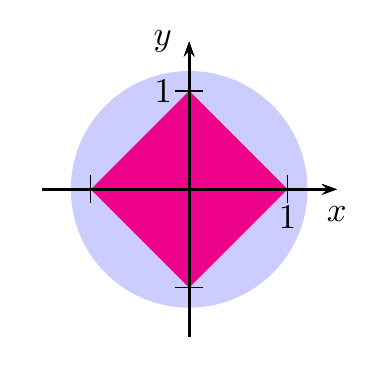} 
\end{center}
\caption{$\{ |x|+|y|<1 \}$ (red) and  $D(0, R_0)$ (blue)}
\label{fig2}
\end{figure}

The proof of Theorem \ref{thmintro} does not rely on the study of a moment problem as in \cite{FR}, or on the duality approach involving some observability inequality for the adjoint problem 
\cite{EZ,FI}. It is based on the flatness approach introduced in \cite{laroche,LMR}  for the motion planning of the one-dimensional heat equation between ``prepared'' states 
(e.g.  the {\em steady} states). Since then, the flatness approach was extended to deal with the null controllability of the heat equation on cylinders, yielding accurate numerical approximations   
of both the controls and the trajectories \cite{MRRautomatica,meurer}, and to give new null controllability 
 results for parabolic equations with discontinuous coefficients that may be degenerate or singular \cite{MRRsiam}. 

For system \eqref{P1}-\eqref{P4} with $h_0\equiv 0$, the flatness approach consists in expressing the solution $\psi $ (resp. the control) in the form
\be
\label{P20}
\psi (x,t)=\sum_{i\ge 0} z^{(i)} (t) \frac{ x^{2i+1} } {(2i+1)! } , \quad h_1(t) =\sum_{i=0}^\infty \frac{ z^{ (i) } (t) }{ (2i+1) ! },
\ee 
where $z\in C^\infty ([0,T])$ is designed in such a way that: (i) the series in \eqref{P20} converges for all $t\in [0,T]$;
(ii) $z^{(i)}(0)=0$ for all $i\ge 0$; and (iii)
\[
\sum_{i=0}^\infty z^{ (i) }(T) \frac{ x^{2i+1} }{ (2i+1)! } =\psi _1(x) \quad \forall x\in (0,1).
\]
If $\psi _1$ is analytic in an open neighborhood of $\{z;  |z|\le 1\}$ and $\psi _1$ is odd, then $\psi _1$ can be written as
\be
\psi _1(x)=\sum_{i=0}^\infty d_i \frac{ x^{2i+1} }{ (2i+1)!}
\label{P25}
\ee
with 
\be
\label{P30}
|d_{i} | \le C \frac{ (2i) !}{R^{ 2i } }
\ee
for some $C>0$  and $R>1$. Thus $\psi _1$ is reachable provided that we can find a  function $z\in C^\infty ([0,T])$ fulfilling the conditions 
\ba
z^{ (i) } (0) &=& 0, \quad \forall i\ge 0, \label{P31}\\
z^{( i) } (T) &=& d_i, \quad \forall i\ge 0, \label{P32}\\
|z^{ (i) } (t) | &\le&  C \left(  \frac{\rho }{R} \right) ^{2i} (2i) ! \quad \forall i\ge 0, \ \forall t\in [0,T], \label{P33} 
\ea
for some constants $C>0$ and $\rho \in (1,R)$.
 
 A famous result due to Borel \cite{borel}   asserts that one can find a function $z\in C^\infty ([0,T])$ satisfying \eqref{P32}. The condition \eqref{P31} can 
 easily be imposed by multiplying $z$ by a convenient cutoff function. Thus, the main difficulty in this approach comes from condition \eqref{P33}, which tells us 
 that the derivatives of the function $z$ (Gevrey of order 2) grow in almost the same way as the $d_i$'s for $t\ne T$.   
 
The Borel interpolation problem in Gevrey classes (or in more general non quasianalytic classes) has been considered e.g. in \cite{CC,LMS,MR,petzsche,ramis,sanz}. 
The existence of a constant $\rho >1$ (that we shall call the {\em loss})  for which \eqref{P33} holds for any $R>0$ and any sequence $(d_i)_{i\ge 0}$ as in 
\eqref{P30}, was proved in those references. Explicit values of $\rho$ were however not provided so far. On the other hand, to the best knowledge of the authors, 
the issue of the determination of the optimal value of $\rho$, 
for any sequence $(d_i)_{i\ge 0}$ or for a given sequence $(d_i)_{i\ge 0}$ as in \eqref{P30}, was not addressed so far. We stress that this issue is crucial here, 
for the convergence of the series in \eqref{P20} requires $R>\rho$: sharp results for the reachable states require sharp results for $\rho$.
  
There are roughly two ways to derive a Borel interpolation theorem in a Gevrey class. 
The {\em complex variable} approach, as e.g. in \cite{LMS,MR,thilliez95}, results in the construction of an interpolating function 
which is complex analytic in a sector of $\C$. It will be used here to derive an interpolation result {\em without loss}, but for a {\em restricted} class of sequences $(d_i)_{i\ge 0} $
(see below Theorem \ref{thm2}). The {\em real variable} approach, as in \cite{AH,petzsche}, yields an infinitely differentiable function of the real variable $x$ only. In \cite{petzsche}, 
Petzsche constructed an interpolating function with the aid of a cut-off function obtained by repeated convolutions of step functions \cite[Thm 1.3.5]{hormander1}.  
Optimizing the constants in Petzsche's construction of the interpolating
function, we shall obtain an interpolation result with as a loss $\rho =R_0\sim 1.2$ (see below Proposition \ref{prop2}).   
 
The paper is outlined as follows.  Section 2 is concerned with the necessary conditions for a state to be reachable (Theorem \ref{thm5}). Section 3 is mainly concerned with
the interpolation problem \eqref{P32}-\eqref{P33}. An interpolation result obtained by the real variable approach 
with as a loss $\rho=R_0\sim 1.2$ is given in Proposition \ref{prop2}. This interpolation result is next applied to the problem 
of the determination of the set of reachable states, first with only one control (the other homogeneous boundary condition being of Neumann or of Dirichlet type), and next with
two boundary controls of Robin type. The section ends with the application of the  complex variable approach to Borel interpolation problem. An interpolation result without loss is derived
(Theorem \ref{thm2}), thanks to which we can exhibit reachable states for \eqref{P1}-\eqref{P4} analytic in $D(1/2,R)$ with $R>1/2$ arbitrarily close to $1/2$. The paper ends with a section providing some 
concluding remarks and some open questions. Two appendices give some additional material. 
   
{\em Notations:}  A function $y\in C^\infty ([t_1,t_2])$ is said to be {\em Gevrey of order $s\ge 0$ on $[t_1,t_2]$} if there exist some constants $C,R>0$ such that 
\[ |y^{ (p)} (t) | \le C \displaystyle\frac{p!^s}{R^p}, \quad \forall p\in \N , \ \forall t\in [t_1,t_2]. \]
The set of functions Gevrey of order $s$ on $[t_1,t_2]$ is denoted $G^s([t_1,t_2])$.
 
A function $\theta \in C^\infty ( [x_1,x_2] \times [t_1,t_2])$ is said to be {\em Gevrey of order $s_1$ in $x$ and $s_2$ in $t$ on $[x_1,x_2]\times [t_1,t_2]$}  if there exist some constants $C,R_1,R_2>0$ such that 
\[
 | \partial _x^{p_1}\partial _t^{p_2} \theta (x,t) |  \le C  \frac{ (p_1 !)^{s_1}  (p_2 !)^{s_2} }{R_1^{p_1} R_2^{p_2}  } \quad \forall p_1,p_2\in \N,\ \forall (x,t) \in [x_1,x_2]\times [t_1,t_2].
\]
The set of functions Gevrey of order $s_1$  in $x$ and $s_2$ in $t$ on $[x_1,x_2]\times [t_1,t_2]$ is denoted $G^{ s_1,s_2 }( [x_1 , x_2] \times [t_1,t_2] )$. 

\section{Necessary conditions for reachability}
In this section, we are interested in deriving necessary conditions for a state function to be reachable from 0. More precisely, we assume given $T>0$ 
and we consider any solution $\psi $ of the heat equation in $(-1,1)\times (0,T)$: 
\begin{equation}
\label{K1}
\psi _t  -\psi _{xx}=0,  \quad x\in (-1,1), \ t\in (0,T).
\end{equation}
Let us introduce the rectangle 
\[
{\mathcal R} :=\{ z=x+iy\in \C ; \ |x|+|y|<1\}. 
\]
The following result gives a necessary condition for a state to be reachable, regardless of the kind of boundary control that is applied.
It extends slightly a classical result due to Gevrey for continuous Dirichlet controls (see \cite{cannon,gevrey}). 
\begin{theorem} 
\label{thm5}
Let $T>0$ and let $\psi$ denote any solution of \eqref{K1}. Then $\psi (.,T')\in H({\mathcal R})$ for all $T'\in (0,T)$. 
\end{theorem}
\begin{proof}
Pick any $\varepsilon >0$ with  $\varepsilon <\min (1, T/2)$.
From \eqref{K1} and a classical interior regularity result (see e.g. \cite[Thm 11.4.12]{hormander2}), we know that 
$\psi \in G^{1,2}( [-1+\varepsilon , 1-\varepsilon ] \times [\varepsilon, T -\varepsilon ])$. Let $h_0(t):=\psi(-1+\varepsilon, t+\varepsilon )$, 
$h_1(t):=\psi(1-\varepsilon, t + \varepsilon)$, $u_0(x):=\psi (x,\varepsilon)$, and $u(x,t):=\psi (x, t+\varepsilon )$. Then 
$u\in G^{1,2} ([-1+\varepsilon , 1-\varepsilon ] \times [0, T-2\varepsilon ])$ is the unique solution
to the following initial-boundary-value problem  
\ba
u_t  -u_{xx}=0, && x\in (-1+\varepsilon, 1-\varepsilon  ), \ t\in (0,T-2\varepsilon ), \label{K11}\\
u(-1 + \varepsilon , t )  =h_0(t), &&t\in (0,T-2\varepsilon ), \label{K12}\\
u(1-\varepsilon ,t)  =h_1(t),&&t\in (0,T-2\varepsilon ), \label{K13}\\
u(x,0)=u_0(x), && x\in (-1+ \varepsilon ,1-\varepsilon ). \label{K14} 
\ea 

Let 
\[
K(x,t):=\frac{1}{ \sqrt{4\pi t} } \exp (-\frac{x^2}{4t} ),\quad x\in \R,\  t>0, 
\]
denote the fundamental solution of the heat equation. By \cite[Theorem 6.5.1]{cannon} (with some obvious change to fit our $x-$domain), 
the solution $u$ of \eqref{K11}-\eqref{K14} 
 can be written as
\[
u(x,t)=  v(x,t) -2\int_0^t \frac{\partial K}{\partial x} (x+1-\varepsilon , t-s)\phi _0 (s) ds   + 2\int_0^t \frac{\partial K}{\partial x} (x-1+\varepsilon , t-s)\phi _1 (s) ds 
\]
where 
\[
v(x,t) =\int_{-\infty} ^\infty K(x-\xi, t)\tilde u_0(\xi)d\xi,
\]
$\tilde u_0$ denoting any smooth, bounded extension of $u_0$ outside of $ -1 + \varepsilon  \le x\le 1 - \varepsilon$, and where 
the pair $(\phi _1,\phi _2)$ solves the system
\ba
h_0(t) &=& v(-1+\varepsilon, t) +  \phi _0 (t) +2\int_0^t  \frac{\partial K}{\partial x} (-2 + 2\varepsilon , t-s) \phi _1 (s) ds,  \label{K21}\\
h_1(t) &=& v(1-\varepsilon, t) + \phi _1 (t) -2\int_0^t  \frac{\partial K}{\partial x} (2-2\varepsilon , t-s) \phi _0 (s) ds. \label{K22}
\ea 
Since $h_0,h_1,v(-1+\varepsilon, .),v(1-\varepsilon, .)\in C([0,T-2\varepsilon ])$, it is well known (see e.g. \cite{cannon}) that the system \eqref{K21}-\eqref{K22}
has a unique solution $(\phi_0,\phi_1)\in C([0,T-2\varepsilon ])^2$. Furthermore, for any $t\in (0,T-2\varepsilon)$, we have that
\begin{itemize}
\item $v(z,t)$ is an entire analytic function in $z$ by \cite[Theorem 10.2.1]{cannon};
\item $\int_0^t \frac{\partial K}{\partial x} (z+1-\varepsilon , t-s)\phi _0 (s) ds$ is analytic in the variable $z$ in the domain $\{ z=x+iy;\ x> -1+\varepsilon, \ |y|< |x+1-\varepsilon | \}$ 
by \cite[Theorem 10.4.1]{cannon};
\item $\int_0^t \frac{\partial K}{\partial x} (z-1+\varepsilon , t-s)\phi _1 (s) ds$ is analytic in the variable $z$ in the domain $\{ z=x+iy;\ x <1-\varepsilon, \ |y|< |x-1+\varepsilon | \}$
by \cite[Corollary 10.4.1]{cannon}.
\end{itemize}
It follows that for $0<t<T-2\varepsilon$,  $z\to u(z,t)$ is analytic in the domain 
\[
{\mathcal R}_\varepsilon :=\{ z=x+iy;\  |x|+|y| <1-\varepsilon \}.
\]
Pick any $T' \in (0,T)$, and pick $\varepsilon <\min (1, T/2, T',T-T')$. Then $T'-\varepsilon \in (0,T-2\varepsilon ) $ and $z\to \psi (z,T')=u(z,T' - \varepsilon )$ is analytic in 
${\mathcal R}_\varepsilon$.   As $\varepsilon$ can be chosen arbitrarily small, we conclude that $z\to \psi (z,T')$ is analytic in $\mathcal R$. 
\end{proof}
Pick any $(\alpha , \beta )\in \R ^2\setminus \{ (0,0) \}$, and  consider the system
\ba
\psi _t -\psi _{xx}=0, \quad && x\in (0,1), \ t\in (0,T), \label{L1}\\
\alpha \psi (1,t) + \beta \psi _x(1,t) =h(t),&&t\in (0,T), \label{L2}\\
\psi (x,0)=\psi _0(x), && x\in (0,1), \label{L3} 
\ea 
supplemented with either the homogeneous Dirichlet condition 
\begin{equation}
\psi (0,t)=0,\quad t\in (0,T),
\label{L4}
\end{equation}
or the homogeneous Neumann condition
\begin{equation}
\psi _x(0,t)=0,\quad t\in (0,T).
\label{L5}
\end{equation}
Then we have the following result. 
\begin{corollary}
\label{cor1}
Let $T>0$, $\psi _0\in L^2(0,1)$ and $h\in L^2(0,T)$. Then the solution $\psi$ of \eqref{L1}-\eqref{L3} and \eqref{L4}  (resp. \eqref{L1}-\eqref{L3}  and \eqref{L5}) is such that  for all $T'\in (0,T)$, the map
$z\to \psi (z,T')$ is analytic in $\mathcal R$ and odd (resp. even).  
\end{corollary}
\begin{proof}
Let $\psi$ denote the solution of \eqref{L1}-\eqref{L3} and \eqref{L4}. Extend $\psi$ to $(-1,1)\times (0,T)$
as an odd function in $x$; that is, set 
\[
\psi (x,t)=\psi(-x,t), \quad  x\in (-1,0),\ t\in (0,T). 
\]
Then it is easily seen that $\psi$ is smooth in $(-1,1)\times (0,T)$ and that it satisfies \eqref{K1}. The conclusion follows from Theorem \ref{thm5}. 
(Note that $\psi (z,T')$ is odd in $z$  for $z\in (-1,1)$, and also  for all $z\in \mathcal R$ by analytic continuation.)
When $\psi$ denotes the solution of \eqref{L1}-\eqref{L3} and \eqref{L5}, we proceed similarly by extending $\psi$ to $(-1,1)\times (0,T)$
as an even function in $x$. 
\end{proof}
\section{Sufficient conditions for reachability}
\subsection{Neumann control}
Let us consider first the Neumann control of the heat equation. We consider the control system
\ba
\theta _t -\theta _{xx}=0, && x\in (0,1), \ t\in (0,T), \label{A1}\\
\theta _x(0,t)=0, \ \theta _x(1,t) =h(t),&&t\in (0,T), \label{A2}\\
\theta (x,0)=\theta _0(x), && x\in (0,1). \label{A3} 
\ea 
We search a solution in the form 
\be
\theta (x,t) = \sum_{i\ge 0} \frac{x^{2i}}{(2i)!} y^{(i)} (t)  \label{A4}
\ee
where $y\in G^2([0,T])$. 
\begin{proposition}
\label{prop1}
Assume that for some constants $M>0$, $R>1$, we have 
\be
|y^{(i)} (t)| \le M \frac{(2i)!}{R^{2i}} \qquad \forall i\ge 0, \ \forall t\in [0,T]. \label{A5} 
\ee 
Then the function $\theta$ given in \eqref{A4} is well defined on $[0,1]\times [0,T]$, and 
\be
\label{D1}
\theta \in G^{1,2}([0,1]\times [0,T]).
\ee
\end{proposition}
\begin{remark}
Proposition \ref{prop1} is sharp as far as the value of $R$ is concerned. It improves some result in \cite[p. 46-47]{laroche}, where the same conclusion was obtained under the assumption 
\be
\label{A5bis}
|y^{(i)} (t)| \le M \frac{i!^2}{\tilde R ^i} \qquad \forall i\ge 0, \ \forall t\in [0,T],
\ee
with $\tilde R>4$. Using Stirling formula, we see that $(2i)!/(i!)^2 \sim 2^{2i} /\sqrt{\pi i}$ so that \eqref{A5} is equivalent to 
\be
\label{A5ter}
|y^{(i)} (t)| \le \frac{M}{\sqrt{\pi i}}  \frac{i !^2}{ (\frac{R^2}{4})^i} \qquad \forall i\ge 0, \ \forall t\in [0,T].
\ee 
Our result is valid whenever $\tilde R:=R^2/4>1/4$. 
\end{remark}
\begin{proof}
We need to prove some uniform estimate for the series of the derivatives
\[
\partial _t^m\partial _x^n [\frac{x^{2i}}{(2i)!} y^{ (i) }(t) ] = \frac{x^{2i-n}}{(2i-n)!} y^{(i+m)} (t)
\]
for $2i-n\ge 0$, $x\in [0,1]$, and $t\in [0,T]$. Fix $m,n\in \N$, $i\in \N$ with $2i-n\ge 0$ and let $j:=2i-n\ge 0$, $N:=n+2m \ge n$
(hence $2i+2m=j+N$). We infer from \eqref{A5} that 
\[\left\vert \frac{y^{(i+m)} (t) }{(2i-n)!}\right\vert \le M \frac{(2i+2m)!}{(2i-n)! \,R^{2i+2m}} =M \frac{(j+N)!}{j! R^{j+N}} \cdot\]
Let
\[
S:=\sum_{i,\, 2i-n\ge 0} \big\vert \partial _t^m\partial _x^n [\frac{x^{2i}}{(2i)!} y^{(i)}  ]\big\vert . 
\]
Then 
\begin{eqnarray*}
S &\le& \sum_{i,\, 2i-n\ge 0} \vert \frac{ y^{(i+m)} (t) }{(2i-n)!}\vert \\
&\le& M\sum_{j\ge 0} \frac{ (j+1) (j+2)\cdots (j+N) }{ R^{j+N} } \\  
&=& M\sum_{ k\ge 0} \ \  \sum_{kN\le j< (k+1)N} \frac{(j+1)(j+2) \cdots (j+N)}{R^{j+N}} \\
&\le & M\sum_{k\ge 0} N\frac{ [ (k+2)N]^N}{ R^{ (k+1)N} } \\
&=& MN^{N+1} \sum_{k\ge 0} \left( \frac{k+2}{ R^{k+1} }\right) ^N.
\end{eqnarray*}
Pick any number $\sigma \in (0,1)$. Since $R>1$, $(k+2)/(R^{1-\sigma})^{k+1} \to 0$ as $k\to  + \infty$, and hence
\[
a:=\sup_{k\ge 0} \frac{k+2}{(R^{1-\sigma})^{k+1}} <\infty. 
\]
We infer that 
\[
\left( \frac{k+2}{R^{k+1}}\right) ^N \le \left( \frac{a}{R^{\sigma (k+1)} }\right) ^N,
\]
and hence 
\[
\sum_{k\ge 0} \left( \frac{k+2}{R^{k+1} }\right) ^N \le a^N \sum_{k\ge 0} R^{ - N\sigma (k+1) } = \frac{a^N}{R^{N\sigma} -1}\cdot
\]
It follows that 
\[
S\le M N^{N+1} \frac{a^N}{ R^{N\sigma } -1} \le M' \left( \frac{ae}{R^\sigma}  \right)^N N! \sqrt{N}
\]
for some constant $M'>0$, where we used Stirling formula $N!\sim (N/e)^N\sqrt{2\pi N}$ in the last inequality. 
Since $N=n+2m$, we have that 
\[
N! \le 2^{n+2m} n! (2m)! \le C \frac{2^{n+4m}}{\sqrt{m}} n! (m!)^2,
\]
where we used again Stirling formula. We conclude that 
\[
S\le CM' \left( \frac{ae}{R^\sigma} \right) ^{n+2m} 2^{n+4m} n! (m!)^2 \sqrt{\frac{n+2m}{m}} \le M'' \frac{(m!)^2}{R_1^m} \frac{n!}{R_2^n} 
\]
for some positive constants $M''$, $R_1$ and $R_2$. (We noticed that $\sqrt{n+2m}\le C \rho ^{n+m}$ for $\rho >1$ and some $C>0$.)
This proves that the series of derivatives $\partial _t^m\partial_x^n [x^{2i} y^{(i)}(t) /(2i)!]$ is uniformly convergent on $[0,1]\times [0,T]$ for all $m,n\ge 0$, so that
$\theta \in C^\infty ([0,1]\times [0,T])$ and it satisfies
\[
|\partial _t^m\partial _x^n \theta (x,t) | \le M'' \frac{(m!)^2}{R_1^m}\frac{n!}{R_2^n} \qquad \forall m,n\in \N, \ \forall (x,t) \in [0,1]\times [0,T],
\]
as desired. The proof of Proposition \ref{prop1} is complete. 
\end{proof}
\begin{theorem}
\label{thm1}
Pick any $\theta _T\in G^1([0,1])$ written as
\be
\theta _T (x) =\sum_{i\ge 0} c_{2i} \frac{ x^{2i} }{ (2i)! }
\label{A11} 
\ee
with
\be
|c_{2i}|\le M\frac{(2i)!}{R^{2i}}
\label{A12} 
\ee
for some $M>0$ and $R>R_0:=e^{(2e)^{-1} } > 1.2$. Then for any $T>0$ and any $R' \in (R_0,R)$, one can pick a function $y\in G^2 ([0,T])$ such that 
\ba 
y^{(i)}(0)=0 \quad &&\forall i\ge 0, \label{A13a}\\
y^{(i)} (T) =c_{2i} \quad &&\forall i\ge 0, \label{A13b}\\
|y^{ (i) } (t) | \le M'  (\frac{R'}{R} )^{2i}  (2i)! \quad && \forall t\in [0,T], \ \forall i\ge 0 \label{A14}
\ea
for some constant $M'>0$. 
Thus, the control input $h(t):=\sum_{i\ge 0} \frac{y^{(i)}(t)}{(2i-1)!}$ is Gevrey of order 2 on $[0,T]$ by Proposition \ref{prop1}, and it steers 
the solution $\theta$ of \eqref{A1}-\eqref{A3} from $0$ at $t=0$ to $\theta _T$ at $t=T$.   
\end{theorem}
\begin{proof}
Using Proposition \ref{prop1}, it is clearly sufficient to prove the existence of a function $y\in C^\infty ([0,T])$ satisfying \eqref{A13a}-\eqref{A14}. 
To do it, we shall need several lemmas. The first one comes from \cite[Theorem 1.3.5]{hormander1}.
\begin{lemma}
\label{lem1}
Let $a_0\ge a_1\ge a_2\ge \cdots >0$ be a sequence such that $a:=\sum_{j=0}^\infty  a_j <\infty$. Then there exists $u\in C_0^\infty (\R )$ such that 
\begin{eqnarray*}
&&u\ge 0, \ \int_{\R} u(x)dx=1,\ \textrm{Supp } u\subset [0,a], \\
\textrm{ and } && |u^{ (k) } (x) |\le 2^k (a_0 a_1 \cdots a_k)^{-1}, \qquad \forall k\ge 0,\ \forall x\in \R.  
\end{eqnarray*} 
\end{lemma}
The following lemma improves slightly Lemma \ref{lem1} as far as the estimates of the derivatives are concerned. 
\begin{lemma}
\label{lem2}
Let $a_0\ge a_1\ge a_2\ge \cdots >0$ be a sequence such that $a:=\sum_{j=0}^\infty  a_j <\infty$. Then for any $\delta >0$,  
there exist $v\in C_0^\infty (\R )$ and $M>0$ such that 
\begin{eqnarray*}
&&v\ge 0, \ \int_{\R} v(x)dx=1,\ \textrm{Supp } v\subset [0,a], \\
\textrm{ and } && |v^{ (k) } (x) |\le  M \delta ^k (a_0 a_1 \cdots a_k)^{-1}, \qquad \forall k\ge 0,\ \forall x\in \R.  
\end{eqnarray*} 
\end{lemma}
\noindent{\em Proof of Lemma \ref{lem2}:} For any given $k_0\in \N$, let 
\[
\kappa := a \big( (k_0+1)a_{k_0} + \sum_{k>k_0} a_k\big) ^{-1},
\]
and 
\[
\tilde a_k := \left\{ 
\begin{array}{ll}
\kappa a _{k_0} &\textrm{ if } 0\le k\le k_0, \\
\kappa a_k &\textrm{ if } k>k_0.
\end{array}
\right. 
\]
Since the sequence $(a_k)_k$ is nonincreasing and the series $\sum_{k\ge 0} a_k$ is convergent, it follows from Pringsheim's theorem (see \cite{hardy}) that 
\[
ka_k\to 0 \qquad \textrm{ as } k\to +\infty . 
\]
Therefore, we may pick $k_0\in \N$ so that 
\[
\kappa > \frac{2}{\delta}\cdot
\] 
Note that $\sum_{k\ge 0} \tilde a_k = a$. Pick a function $v\in C_0^\infty(\R )$ as in Lemma \ref{lem1} and associated with the sequence $(\tilde a_k)_{k\ge 0}$; that is, 
\begin{eqnarray*}
&&v\ge 0, \ \int_{\R} v(x)dx=1,\ \textrm{Supp } v\subset [0,a], \\
\textrm{ and } && |v^{ (k) } (x) |\le 2^k (\tilde a_0 \tilde a_1 \cdots \tilde a_k)^{-1}, \qquad \forall k\ge 0,\ \forall x\in \R.  
\end{eqnarray*} 
Then, for any $k\ge k_0$,
\begin{eqnarray*}
|v^{(k)} (x) | &\le& (\kappa a_{k_0} )^{-(k_0+1)} 2^k \kappa ^{-(k-k_0)} (a_{k_0+1} \cdots a_k)^{-1} \\
&\le & \kappa ^{-1} (\frac{2}{\kappa}) ^k \frac{a_0a_1\cdots a_{k_0}}{a_{k_0}^{k_0+1}} \big( \prod_{i=0}^k a_i \big) ^{-1}. 
\end{eqnarray*} 
Thus
\[
|v^{(k)} (x) | \le M_1 \delta ^k (a_0\cdots a_k)^{-1}, 
\]
where 
\[M_1 := \kappa ^{-1} \frac{a_0a_1\cdots a_{k_0-1}}{a_{k_0}^{k_0}}\cdot \]
Finally, for $0\le k\le k_0$ and $x\in \R$ 
\[
|v^{(k)} (x)| \le 2^k (\kappa a_{k_0} )^{-(k+1)}  \le M_2 \delta ^k (a_0\cdots a_k)^{-1}
\]
with 
\[
M_2:= \sup_{0\le k\le k_0} (\frac{2}{\delta})^k \frac{a_0a_1\cdots a_k}{(\kappa a_{k_0})^{k+1}}\cdot
\]
We conclude that 
\[
|v^{(k)} (x)| \le M \delta ^k (a_0\cdots a_k)^{-1}, \qquad \forall k\ge 0,\ \forall x\in \R 
\]
with $M:=\sup (M_1,M_2)$.\qed

\begin{corollary}
\label{cor2}
For any sequence $(a_k)_{k\ge 1}$ satisfying $a_1\ge a_2\ge \cdots >0$ and  $\sum_{k\ge 1} a_k<\infty$ and for any $\delta >0$, there exists 
a function $\varphi\in C^\infty _0(\R )$ with $\textrm{Supp } \varphi \subset [-a,a]$, $0\le \varphi \le 1$, $\varphi ^{(p)} (0)=\delta _ p^0$ and 
\[
| \varphi ^{( k) }  (x)| \le C \delta ^k (a_1 \cdots a_k)^{-1}\quad \forall k\ge 0, \ \forall x\in \R ,
\]
with the convention that $a_1\cdots a_k=1$ if $k=0$.  
\end{corollary}
Indeed, there exist by Lemma \ref{lem2} a function $v\in C_0^\infty (\R )$ and a number $C>0$ such that 
\begin{eqnarray*}
&&v\ge 0, \ \int_{\R} v(x)dx=1,\ \textrm{Supp } v\subset [0,a], \\
\textrm{ and } && |v^{ (k) } (x) |\le C\delta ^k (a_1  \cdots a_{k+1})^{-1}, \qquad \forall k\ge 0,\ \forall x\in \R.  
\end{eqnarray*} 
Note that $v^{(k)} (a) =0$ $ \forall k\ge 0$. Let 
\[
\varphi (x) := \int_{-\infty} ^{a-|x|} v(s) ds.
\]
Clearly, $\varphi \in C_0^\infty (\R )$, $0\le \varphi \le 1$, $\textrm{ Supp } \varphi \subset [-a,a]$, $\varphi ^{(j)} (0)=\delta _j^0$, and 
\[
|\varphi ^{(k)} (x) | \le C \frac{\delta ^{k-1}}{ a_1  \cdots a_{k}}
\le C' \frac{\delta ^{k}}{ a_1  \cdots a_{k}}\qquad \forall k\ge 1, \ \forall x\in \R .
\]

\begin{proposition} \label{prop2}
Pick any sequence $(a_k)_{k\ge 0}$ satisfying $1=a_0\ge a_1\ge a_2\ge \cdots >0$ and
\begin{eqnarray*}
&&a:=\sum_{k\ge 1} a_k <\infty, \\
&& pa_p + \sum_{k>p} a_k \le  A pa_p\qquad \forall p\ge 1, 
\end{eqnarray*}  
for some constant $A\in (0,+\infty )$. Let $M_q:=(a_0\cdots a_q)^{-1}$ for $q\ge 0$. 
Then for any sequence of real numbers $(d_q)_{q \ge 0}$ such that 
\be
\label{B1}
|d_q| \le C H^q M_q \quad  \forall q\ge 0, 
\ee
for some $H>0$ and $C>0$, and for any $\tilde H>e^{e^{-1}}H$, 
there exists a function $f\in C^\infty (\R )$ such that 
\ba
f^{(q)}(0)&=&d_q,\quad  \forall q\ge 0,  \label{B2}\\
|f^{(q)} (x) | &\le& C\tilde H^q M_q \quad \forall q\ge 0, \ \forall x\in\R . \label{B3}
\ea 
\end{proposition}
\begin{proof}
We follow closely \cite{petzsche}. 
Let $m_q := 1/a_q$  for $q\ge 0$, so that $M_q=m_0\cdots m_q$. 
For any given $h>0$, we set $\tilde a_k := h^{-1} a_k$ for all $k\ge 0$,  so that 
\begin{eqnarray*}
&& h^{-1} a=\sum_{k\ge 1} \tilde a_k , \\
&& p \tilde a_p + \sum_{k>p} \tilde a_k \le  A p \tilde a_p = Ap/(hm_p)\qquad \forall p\ge 1. 
\end{eqnarray*} 
By Corollary \ref{cor2} applied to the sequence $(\tilde a_q)_{q\ge 1}$, there is a function $\varphi\in C^\infty_0(\R )$ such that 
$\textrm{Supp } \varphi \subset [-h^{-1} a, h^{-1}a]$, $0\le \varphi \le 1$,  $\varphi ^{(j)} (0)=\delta _ j^0$ and 
\[
| \varphi  ^{( j) }  (x) | \le C (\delta h) ^j M_j\qquad \forall j\ge 0.  
\]
Set $\zeta _0(x)=\varphi_0(x):=\varphi(x)$. 
Applying for any  $p\in \N ^*$ Corollary \ref{cor2} to the sequence 
\[
\hat a_k := \left\{ 
\begin{array}{ll}
\tilde a_p &\textrm{ if } 1\le k\le p,\\
\tilde a_k &\textrm{ if } k\ge p+1,
\end{array}
\right.
\]
we may also pick a function $\varphi _p\in C^\infty(\R )$ with 
$\varphi _p \subset [-Ap/(hm_p),Ap/(hm_p)]$, $0\le \varphi _p \le 1$, $\varphi _p^{(j)} (0)=\delta _ j^0$ and 
\[
| \varphi _p^{( j) }  (x)| \le  
\left\{ 
\begin{array}{ll}
C(\delta  hm_p)^j &\textrm{ if }  0\le j\le p,\\
C (\delta h)^j m_p^p \frac{M_j}{M_p}& \textrm{ if } j>p.   
\end{array}
\right.
\]
We set $\zeta _p(x):=\varphi _p(x) \frac{x^p}{ p! }$, so that  $\zeta _p^{ (j) }(0)=\delta _j^p$. To estimate $\zeta _p^{(j)}$ for $p\ge 1$ and $j\ge 0$, we distinguish two cases.\\
(i) For $0\le j\le p$, we have
\begin{eqnarray*}
|\zeta _p^{(j)} (x) | 
&\le& \sum_{i=0}^j \left(\begin{array}{c} j\\i \end{array}\right) |\varphi _p^{(i)} (x) \frac{x^{p-j+i}}{(p-j+i)!}| \\
&\le& C\sum_{i=0}^j \left(\begin{array}{c} j\\i \end{array}\right) (\delta  h m_p)^i \left( \frac{A}{ h m_p} \right)^{p-j +i}  \frac{p^{p-j+i}}{(p-j+i)!} \\
&\le& C \frac{M_j}{M_p} (\frac{A}{h})^p \sum_{i=0}^j \left(\begin{array}{c} j\\i \end{array}\right) (\delta h) ^i  \left( \frac{h}{A} \right)^{j-i}  \frac{m_{j+1}\cdots m_p}{m_p^{p-j}} 
  \frac{p^{p-j+i}}{(p-j+i)!}  \\
&\le&  C \frac{M_j}{M_p} (\frac{Ae}{h})^p \sum_{i=0}^j \left(\begin{array}{c} j\\i \end{array}\right) 
(\delta h) ^i  \left( \frac{h}{Ae} \right)^{j-i}    \left( \frac{ p}{ p-j+i } \right) ^{p-j+i}
\end{eqnarray*} 
where we used Stirling's formula and the fact that the sequence $(m_i)_{i\ge 0}$ is nondecreasing in the last inequality. Elementary
computations  show that the function 
$x \in [0,p] \to \left( \frac{p}{p-x} \right) ^{p-x} \in \R$ reaches its greatest value for $x=p(1-e ^{-1})$, and hence 
\[
\left( \frac{ p}{ p-j+i } \right) ^{p-j+i} \le e^{ \frac{p}{e} }.
\] 
We conclude that 
\be
|\zeta _p^{(j)} (x) | \le  C \frac{M_j}{M_p} (\frac{Ae^{1+e^{-1}}}{h})^p  \left( \delta h + \frac{h}{Ae} \right)^j. 
\label{A100}
\ee
(ii) For $j > p$, we have
\begin{eqnarray*}
|\zeta _p^{(j)} (x) | 
&\le& \sum_{i=j-p}^j \left(\begin{array}{c} j\\i \end{array}\right) |\varphi _p^{(i)} (x) \frac{x^{p-j+i}}{(p-j+i)!}| \\
&\le& C\sum_{ \tiny \begin{array}{c} j-p\le i\le j\\ \textrm{and } i\le p \end{array}  } \left(\begin{array}{c} j\\i \end{array}\right) (\delta  h m_p)^i \left( \frac{A}{ h m_p} \right)^{p-j +i}  \frac{p^{p-j+i}}{(p-j+i)!} \\
&&+\,  C\sum_{\tiny \begin{array}{c} j-p\le i\le j\\ \textrm{and } i\ge  p+1 \end{array} }\left(\begin{array}{c} j\\i \end{array}\right) (\delta  h)^i  m_p^p \frac{M_i}{M_p}\left( \frac{A}{ h m_p} \right)^{p-j +i}  \frac{p^{p-j+i}}{(p-j+i)!} =:CS_1+CS_2. \\
\end{eqnarray*} 
We infer from the computations in the case $0\le j\le p$ that 
\[
S_1 \le C \frac{M_j}{M_p} (\frac{Ae^{1+e^{-1}}}{h})^p  \left( \delta h + \frac{h}{Ae} \right)^j. 
\]
For $S_2$, we notice that 
\[
\frac{M_i}{M_p} m_p^{j-i}  
= \frac{M_j}{M_p} \frac{m_p^{j-i} m_{p+1} \cdots m_i}{ m_{p+1}\cdots m_j}   \le  \frac{M_j}{M_p}
\]
where we used again the fact that the sequence $(m_j)_{j\ge 0}$ is nondecreasing.
Thus 
\begin{eqnarray*}
S_2 &\le & 
C \frac{M_j}{M_p} (\frac{Ae}{h})^p \sum_{ \tiny \begin{array}{c} j-p\le i\le j\\ \textrm{and } i\ge  p+1 \end{array}    } \left(\begin{array}{c} j\\i \end{array}\right) 
(\delta h) ^i  \left( \frac{h}{Ae} \right)^{j-i}    \left( \frac{ p}{ p-j+i } \right) ^{p-j+i} \\
&\le& 
C \frac{M_j}{M_p} (\frac{Ae^{1+e^{-1}} }{h})^p 
\left(  \delta h  + \frac{h}{Ae} \right)^j.  
\end{eqnarray*}
We conclude that \eqref{A100} is valid also for $j>p$. Clearly, \eqref{A100} is also true for $p=0$ and $j\ge 0$.  

Let the sequence $(d_q)_{q\ge 0}$ and the number $H$ be as in \eqref{B1}. Pick $\delta >0$ and $h>0$ such that 
\[
h= (1+\delta ) Ae^{1+e^{-1}} H,
\]
and 
\[
\delta h + \frac{h}{Ae}= (1+\delta) (\delta A e + 1) e^{e^{-1}}H <\tilde H. 
\] 
Let $f(x) := \sum_{p\ge 0} d_p\zeta _p(x)$. Then $f\in C^\infty( \R )$, and for all $x\in \R$ and all $j \ge 0$  
\begin{eqnarray*}
| f^{ (j) } (x)| &\le & \sum_{p\ge 0} |d_p \zeta _p^{(j)} (x) | \\
&\le & C \sum_{p\ge 0}  (H^p M_p) \frac{M_j}{M_p}  (\frac{Ae^{1+e^{-1}} }{h})^p  \left(  \delta h  + \frac{h}{Ae} \right)^j \\
&\le& C\sum_{p\ge 1} (1+\delta )^{-p} \tilde H ^jM_j \\
&\le& C\tilde H ^jM_j.
\end{eqnarray*}
Thus \eqref{B3} holds, and \eqref{B2} is obvious. 

To ensure that the support of $f$ can be chosen as small as desired, we need the following 
\begin{lemma}
\label{lem3}
Let $-\infty < T_1<T_2<\infty$, $1<\sigma <s$,  and let $f\in G^s([T_1,T_2])$ and  $g\in G^\sigma ([T_1,T_2])$; that is
\begin{eqnarray}
| f^{ (n) } (t) | &\le& C \frac{n!^s }{R^n} \qquad \forall t\in [T_1,T_2], \ \forall n\ge 0, \label{E1} \\
| g^{ (n) } (t) | &\le& C' \frac{n!^\sigma  }{\rho ^n} \qquad \forall t\in [T_1,T_2], \ \forall n\ge 0, \label{E2}
\end{eqnarray} 
where $C,C',R$ and $ \rho$ are some positive constants. Then $fg\in G^s ([T_1,T_2])$ with the same $R$ as for $f$; that is, we have for some
constant $C''>0$
\begin{equation}
|(fg)^{ (n) } (t) | \le C'' \frac{n! ^s}{R ^n} \qquad \forall t\in [T_1,T_2], \ \forall n\ge 0. 
\end{equation}
\end{lemma}
\begin{proof}
From Leibniz' rule and \eqref{E1}-\eqref{E2}, we have that
\begin{equation}
|(fg)^{(n)} (t) |  = \left\vert \sum_{j=0} ^ n 
\left(
\begin{array}{c} 
n\\
j
\end{array} 
\right) 
f^{ (j) } (t) g^{ (n-j) } (t) 
\right\vert 
\le CC' \sum_{j=0}^n 
\left(
\begin{array}{c} 
n\\
j
\end{array} 
\right) 
\frac{j! ^s (n-j) !^\sigma}{R^j \rho ^{n-j}}.
\label{E3}
\end{equation}

We claim that 
\begin{equation}
\left(
\begin{array}{c} 
n\\
j
\end{array} 
\right) 
\frac{j! ^s (n-j) !^\sigma}{R^j \rho ^{n-j}}
\le 
\tilde C \frac{ n! ^s }{ 2^{ n- j } R^n}\cdot \label{E4}
\end{equation}

Indeed, \eqref{E4} is equivalent to 
\begin{equation}
j! ^{s-1} (n-j) ! ^{\sigma -1} 
\le \tilde C ( \frac{\rho }{2R} )^{ n- j } n! ^{s-1} \label{E5}
\end{equation}
and to prove \eqref{E5}, we note that, since $ 1< \sigma  < s$, 
\[
j! ^{s-1} (n-j) ! ^{\sigma -1}  = (n-j) ! ^{\sigma -s } (j! (n-j) ! )^{s-1} 
\le \tilde C ( \frac{\rho }{2R} )^{ n- j } n! ^{s-1} 
\]
for some constant $\tilde C>0$ and all $0\le j\le n$.  It follows from \eqref{E3}-\eqref{E4} that 
\[
|(fg) ^{(n)} (t) | \le 2CC'\tilde C \frac{n!^s}{R^n}\qquad \forall t\in [T_1,T_2], \ \forall n\ge 0. 
\]
\end{proof}

Multiplying $f$ by a cutoff function $g\in C^\infty_0(\R ) $ with $g(x)=1$ for $|x|<a/2$, $\textrm{Supp } (g)\subset [-a,a]$, and 
$g\in G^\sigma ([-a,a])$ for some  $1<\sigma <2$,  
  we can  still assume that $\textrm{Supp } f \subset [-a,a]$.  The proof of Proposition \ref{prop2} is complete. \end{proof}
We apply Proposition \ref{prop2} with $d_q=c_{2q}$ for all $q\ge 0$, $a_0=1$, $a_p=[2p(2p-1)]^{-1}$ for $p\ge 1$, so that $M_p=(2p)!$, $H=R^{-2}$, 
$  \tilde H \in (e^{e^{-1}}H, (R'/R)^2 )$. Let $f$ be as in \eqref{B2}-\eqref{B3}, and pick $g\in G^\sigma ([0,T])$ with $1<\sigma <2$, $g^{(i)}(0)=0$ for all $i\ge 0$, 
$g(T)=1$, and $g^{(i)}(T)=0$ for all $i\ge 1$. Set finally 
\[
y(t)=f(t -T)g(t), \quad t\in [0,T]. 
\]
Since by \eqref{B3}
\[
| f^{(i)} (t) | \le C\tilde H ^i (2i) ! \le C \frac{ (4\tilde H )^i i! ^2 }{\sqrt{i+1} }\quad \forall t\in \R,
\]
we infer from Lemma \ref{lem3} that 
\[
| y^{ (i) } (t) | \le C (4\tilde H )^i i! ^2 \le C\sqrt{i+1} \tilde H ^i (2i) ! \le C ( \frac{R'}{R} )^{2i} (2i) !\qquad \forall i\in [0,T], 
\]
i.e. \eqref{A14} holds. The properties \eqref{A13a} and \eqref{A13b} are clearly satisfied. The proof of Theorem \ref{thm1} is complete.  
\end{proof}
The following result, proved by using the {\em complex variable approach}, gives a Borel interpolation result without loss, but for a restricted class of sequences 
$(d_n)_{n\ge 0}$. 
\begin{theorem}
\label{thm2}
Let $(d_n)_{n\ge 0} $ be a sequence of complex numbers such that 
\be
\label{Y1}
|d_n| \le M \frac{(2n)!}{R^{2n}}\quad \forall n\in \N
\ee
for some constants $M>0$ and $R>1$, and such that the function 
\be
\label{Y2}
g(z):= \sum_{n\ge 0} \frac{d_{n+1}}{n! (n+1)!} z^n, \quad |z|<R^2/4,
\ee
can be extended as an analytic function in an open  neighborhood of $\R _-$ in $\C$  with 
\be
\label{Y3}
|g^{(n)} (z) |\le C|g^{ (n) } (0)| \quad \forall z\in \R_-,\ \forall n\in \N 
\ee
for some constant $C>0$.  
Let $-\infty < T_1< T_2<\infty$. 
Then there exists a function $f\in C^\infty ( [ T_1, T_2],\C )$ such that 
\ba
&&f^{(n)} ( T_1)=0 \qquad \forall n\ge 0, \label{Y3a}\\
&&f^{(n)} ( T_2)=d_n \qquad \forall n\ge 0, \label{Y3b}\\
&&| f^{ (n) } (t)| \le M' \frac{ (2n)! }{ R^{2n} } \qquad  \forall n\ge 0, \ \forall t\in [ T_1 , T_2] \label{Y4}
\ea
for some $M'>0$ and the same constant $R>1$ as in \eqref{Y1}. 
\end{theorem}
\begin{proof}
Applying a translation in the variable $t$, we may assume that $T_2=0$ without loss of generality. 
Let 
\[
f(t) =d_0+ \int_0^{-\infty} e^{-\xi/t} g(\xi ) d\xi,\quad  t<0.
\]
By \eqref{Y3}, $f$ is well defined and 
\[
|f(t)-d_0| \le C| g(0)  \int _0^{-\infty} e^{-\xi/t}  d\xi | =  C|tg(0)|,
\]
so that $f(0^-)=d_0$. Applying Lebesgue dominated convergence theorem and a change of variables, we infer that
\[
f'(t) = \int_0^{-\infty} e^{-\xi /t} \xi t^{-2} g (\xi ) d\xi = \int_0^\infty e^{-s} g(ts) sds.  
\] 
We obtain by an easy induction that for $n\ge 1$ and $t<0$
\be
\label{Y5}
f ^{(n)} (t ) =\int_0^\infty e^{-s} g ^{(n-1)} (ts) s^n ds. 
\ee
Thus $f\in C^\infty ((-\infty , 0])$, and using \eqref{Y2} we obtain
\be
\label{Y6}
f^{(n)}(0)=d_n \quad \forall n\in \N .
\ee 
On the other hand, we have by \eqref{Y1}, \eqref{Y3} and \eqref{Y5}-\eqref{Y6} that
\[
| f ^{ (n) } (t) | \le C | f^{ (n) } (0) | \le MC \frac{ (2n)! }{ R^{2n} } \cdot 
\]
Multiplying $f$ by a function $g\in G^\sigma ([T_1,0])$ with $1<\sigma <2$ such that  $g^{(n)}(T_1)= g^{(n)}(0)=0$ for $n\ge 1$,
$g(T_1)=0$,  and $g(0)=1$, we infer from a slight modification of the proof of Lemma \ref{lem3} that $fg$ satisfies \eqref{Y3a}-\eqref{Y4}. 
The proof of Theorem \ref{thm2} is complete.
\end{proof}
\noindent
\begin{remark}
\begin{enumerate}
\item Theorem \ref{thm2} shows that for certain sequences $(d_n)_{n\ge 0} $, Borel Theorem can be established {\em without} any loss 
in the factor $R$. It requires the (quite conservative) assumption \eqref{Y3}. Recall (see \cite{rudin}) that for any open set $\Omega\subset \C$,  one can find a function
$g\in H(\Omega )$ which has no holomorphic extension to any larger region.  
\item Condition \eqref{Y3} is satisfied e.g. for $g(z) := \exp (z)$. This corresponds to the sequence  $d_n=n!$ for all $n\ge 0$. 
\item Condition \eqref{Y3} is also satisfied for 
\[
g(z) := (z-z_0)^{-k},
\]
when $k\in \N ^*$ and  $z_0\in \C_+ :=\{ z = x + i  y ,\ x>0 \}$. Indeed, 
\[
g^{(n)} (z) =(-1)^n \frac{k(k+1)\cdots (k+n -1)}{(z-z_0)^{k+n} }
\]
and hence
\[
| g^{ (n) } (z) |  \le | g^{ (n) } (0)| ,\quad \forall z\in \R _-,\ \forall n\in \N ,   
\]
since $|z-z_0|\ge |z_0|$ for $z\in \R _-$ and $n\in \N$. 
\end{enumerate}
\end{remark}
\begin{corollary} 
\label{cor3}
Let $T>0$ and 
\be
\label{Y10}
\theta _T (x) := \sum_{n\ge 0} d_n \frac{ x^{2n} }{ (2n) ! }
\ee
with a sequence $(d_n)_{n\ge 0}$ as in Theorem \ref{thm2}. Take $(T_1,T_2)=(0,T)$, and pick a function $f$ as in Theorem \ref{thm2}.  
Then the control input $h(t):=\sum_{ n\ge 0} \frac{f^{(n)}(t)}{(2n-1)!}$ is Gevrey of order 2 on $[0,T]$, and it steers 
the solution $\theta$ of \eqref{A1}-\eqref{A3} from $0$ at $t=0$ to $\theta _T$ at $t=T$.   
\end{corollary}

As an example of application of Corollary  \ref{cor3}, pick any $\zeta =r e^{i\theta } \in \C$ with $r>1/2$ and $ | \theta  | <\pi /4$, and let
\[
g(z) :=  \zeta ^{-2}(1-\frac{z}{\zeta ^2} )^{-2} = \sum_{n\ge 0} \frac{d_{n+1}}{ n! (n+1)! } z^n 
\] 
where 
\[
d_n:= \frac{ (n !)^2 }{ \zeta ^{2n} } \cdot
\]
Then \eqref{Y1}-\eqref{Y3} hold with $C=1$, any $R\in (1, 2 |\zeta |)$, and some $M>1$. Thus the  state $\theta _T$  given in 
\eqref{Y10} is reachable from 0 in time $T$. Note that the radius of convergence of the series in \eqref{Y10} can  be chosen arbitrarily close to 1.  
\subsection{Dirichlet control} 
Let us turn now our attention to the Dirichlet   control of the heat equation. We consider the control system
\ba
\phi _t -\phi _{xx}=0, \quad && x\in (0,1), \ t\in (0,T), \label{EE1}\\
\phi (0,t)=0, \ \phi (1,t) =k(t),&&t\in (0,T), \label{EE2}\\
\phi (x,0)=\phi _0(x), && x\in (0,1). \label{EE3} 
\ea 
We search a solution in the form 
\be
\phi (x,t) = \sum_{i\ge 0} \frac{x^{2i+1 } }{(2i+1)!} z^{(i)} (t)  \label{EE4}
\ee
where $z\in G^2 ([0,T])$. The following result is proved in exactly the same way as Proposition \ref{prop1}. 
\begin{proposition}
\label{prop11}
Assume that for some constants $M>0$, $R>1$, we have 
\be
|z^{(i)} (t)| \le M \frac{(2i +1)!}{R^{2i+1}} \qquad \forall i\ge 0, \ \forall t\in [0,T]. \label{EE5} 
\ee 
Then the function $\phi$ given in \eqref{EE4} is well defined on $[0,1]\times [0,T]$, and 
\be
\label{EE6}
\phi \in G^{1,2}([0,1]\times [0,T]).
\ee
\end{proposition}
With Proposition \ref{prop11} at hand, we can obtain the following
\begin{theorem}
\label{thm11} 
Pick any $\phi _T\in G^1([0,1])$ written as
\be
\phi _T (x) =\sum_{i\ge 0} c_{2i+1} \frac{ x^{2i+1} }{ (2i +1 )! }
\label{EE7} 
\ee
with
\be
 | c_{2i +1 } | \le M\frac{(2i +1 )!}{R^{2i +1 } }
\label{EE8} 
\ee
for some $M>0$ and $R>R_0$. Then for any $T>0$ and any $R' \in (R_0,R)$, one can pick a function $z\in G^2 ([0,T])$ such that 
\ba 
z^{(i)}(0)=0 \quad &&\forall i\ge 0, \label{EE9}\\
z^{(i)} (T) =c_{2i + 1} \quad &&\forall i\ge 0, \label{EE10}\\
|z^{ (i) } (t) | \le M'  (\frac{R'}{R} )^{2i + 1 }  (2i + 1 )! \quad && \forall t\in [0,T], \ \forall i\ge 0. \label{EE11}
\ea
Thus, the control input $k(t):=\sum_{i\ge 0} \frac{z^{(i)}(t)}{(2i+1)!}$ is Gevrey of order 2 on $[0,T]$ by Proposition \ref{prop11}, and it steers 
the solution $\phi$ of \eqref{EE1}-\eqref{EE3} from $0$ at $t=0$ to $\phi _T$ at $t=T$.   
\end{theorem}
\begin{proof}
Pick $\tilde R$ and  $\tilde R'$ such that 
\[
R_0 < \tilde  R' <  R' < \tilde R < R \quad \textrm{ and } \frac{\tilde R'}{\tilde R} < \frac{R'}{R}  \cdot
\]
Then, from \eqref{EE8}, we have that for some constant $\tilde M>0$:
\[
|c_{2i+1} | \le \tilde M \frac{ (2i) !} { \tilde R ^{2i} }\cdot
\]
From the proof of Theorem \ref{thm1}, we infer the existence of $z\in C^\infty ([0,T])$ such that \eqref{EE9}-\eqref{EE10} hold and such that we have
for some $\tilde M>0$ 
\begin{equation}
| z^{ (i) } (t) | \le \tilde M (\frac{\tilde R'}{\tilde R }) ^{2i} (2i)!  \le \tilde M (\frac{ R' }{ R } ) ^{2i} (2i)! \quad \forall t\in [0,T], \ \forall i\ge 0.
\label{EE12} 
\end{equation} 
Then \eqref{EE11} follows from \eqref{EE12} by letting $M':=\tilde M R/R'$. 
\end{proof}
\subsection{Two-sided control}
Let $(\alpha_0,\beta _0), (\alpha _1, \beta _1)\in \R ^2\setminus \{ (0,0) \} $.  We are concerned here with the control problem:
\ba
\psi _t  -\psi _{xx}=0, && x\in (-1,1), \ t\in (0,T), \label{F1}\\
\alpha _0\psi (-1,t) + \beta _0 \psi _x (-1,t) =h_0(t),&&t\in (0,T), \label{F2}\\
\alpha _1\psi (1,t) + \beta _1 \psi _x (1,t) =h_1(t),&&t\in (0,T), \label{F3}\\
\psi (x,0)=0, && x\in (-1,1). \label{F4} 
\ea Then the following result holds.
\begin{theorem}
Let $T>0$ and $R>R_0$. Pick any $\psi _T\in H(D(0,R))$. Then one may find two control functions $h_0,h_1\in G^2([0,T])$ such that the solution $\psi$ of 
\eqref{F1}-\eqref{F4} belongs to $G^{1,2}([-1,1]\times [0,T])$ and it satisfies 
\be
\psi (x,T)=\psi_T(x)\quad \forall x\in [-1,1].
\ee 
\end{theorem}
\begin{proof}
Since $\psi _T\in H(D(0,R))$, $\psi _T$ can be expanded as 
\[
\psi _T (z) = \sum_{i\ge 0} c_i \frac{ z^i }{ i !} \qquad \textrm{ for } |z| <R
\] 
where $c_i:=\psi _T^{ (i) } (0)$ for all $i\ge 0$.   Pick any $R'\in (R_0,R)$. It follows from the Cauchy inequality (see e..g \cite{titchmarsh}) that  
\[
|c_i| \le \sup_{|z| \le R'} |\psi _T (z) | \frac{ i ! } {{R'} ^i} =: M  \frac{ i ! } { {R'}^i}\cdot 
\]
Let 
\[
\theta _T (x) := \sum _{i\ge 0} c_{2i} \frac{ x^{2i } }{ (2i)! } \textrm { and } 
\phi _T (x) := \sum _{i\ge 0} c_{2i +1} \frac{ x^{2i +1} }{ ( 2i +1) ! }  \textrm{ for } x\in (-R',R'). 
\]
Let $y$ (resp. $z$) be as given by Theorem \ref{thm1} (resp. Theorem \ref{thm11}), and let 
\[
\psi (x,t) := \theta (x,t) + \phi (x,t) = \sum_{ i\ge 0 } \frac{ x^{2i} }{ (2i)! } y^{ (i) } (t) + \sum_{ i\ge 0 } \frac{ x^{2i+1} }{ (2i +1)! } z^{ (i) } (t), \quad \textrm{ for } 
x\in [0,1],\ t\in [0,T]. 
\]
It follows from  Theorems \ref{thm1} and \ref{thm11} that $\psi \in G^{1,2}([0,1]\times [0,T])$. Note that the function $\theta$ (resp. $\phi$) can be extended
as a smooth function on $[-1,1]\times [0,T]$ which is even (resp. odd) with respect to $x$. Thus  $\theta, \phi, \psi  \in G^{1,2}([-1,1]\times [0,T])$.
Define $h_0$ and $h_1$ by \eqref{F2} and \eqref{F3}, respectively. Then $h_0,h_1\in G^2([0,T])$, and $\psi$ solves \eqref{F1}-\eqref{F4} together with 
 \[
 \psi (x,T)=\theta _T (x) + \phi _T (x) =\psi _T(x), \quad \forall x\in [-1,1]. 
 \]
\end{proof}
\begin{remark}
\begin{enumerate}
\item The control functions $h_0,h_1$ take complex values if the target function $\psi_T$ does on $(-1,1)$ (i.e. if at least one $c_i\in \C$). If, instead, $\psi_T$
takes real values on $(-1,1)$, then we can as well impose that both control functions $h_0,h_1$ take real values by extracting the real part of each term in \eqref{F1}-\eqref{F3}. 
\item 
For any given $a\in \R_+$, let 
\[
\psi _T(z) := \frac{1}{z^2+a^2}\quad \textrm{ for } z\in D(0,|a|). 
\]
Then $\psi _T\in H(D(0,|a|))$, and for $|a| > R_0$, we can find a pair of control functions $(h_0,h_1) \in G^2 ([0,T]) ^2$ driving the solution 
of \eqref{F1}-\eqref{F4} to $\psi _T$ at $t=T$. If $(\alpha _0,\beta _0)=(0,1)$, then it follows from Theorem \ref{thm1} that we can reach $\psi_T$ on $(0,1)$  
with only one control, namely $h_1$  (letting $\psi _x(0,t)=0$). Indeed, both conditions \eqref{A11} and \eqref{A12} are satisfied, for $\psi_T$ is analytic in 
$H(D(0,|a|))$ and even. Similarly, If $(\alpha _0,\beta _0)=(1,0)$, then it follows from Theorem \ref{thm11} that we can reach $x\psi_T(x)$ on $(0,1)$ 
with solely the control $h_1$. 

On the other hand, if  $|a|<1$, then by Theorem \ref{thm5} there is no pair $(h_0,h_1)\in L^2(0,T)^2$ 
driving the solution  of \eqref{F1}-\eqref{F4}  to $\psi _T$ at $t=T$. 
\end{enumerate}
\end{remark}
\section{Further comments}
In this paper, we showed that for the boundary control of the one-dimensional heat equation, the reachable states are those functions that can be extended as (complex) analytic
functions on a sufficiently large domain. In particular, for the system 
\begin{eqnarray}
\psi _t-\psi _{xx} =0,&& \quad x\in (-1,1), \ t\in (0,T), \label{LL1}\\
\psi (-1,t)=h_0(t),&&\quad t\in (0,T), \label{LL2}\\
\psi (1,t)=h_1(t),&& \quad t\in (0,T), \label{LL3}\\
\psi (x,0)=\psi _0(x)=0,&&  \quad x\in (-1,1),\label{LL4}
\end{eqnarray} 
the set of reachable states 
\[
{\mathcal R}_T :=\{ \psi (.,T);\ h_0,h_1\in L^2(0,T) \} 
\]
satisfies 
\[
H ( D(0,R_0) ) \subset {\mathcal R}_T  \subset H ( \{ z=x+iy;\ |x|+|y| <1 \} ) 
\]
where $R_0:=e^{ (2e)^{-1}}>1.2$. 

Below are some open questions: 
\begin{enumerate}
\item Do we have for any $R>1$
\[
H(D(0,R)) \subset {\mathcal R}_T?
\]
\item Do we have 
\[
{\mathcal R}_T \subset H(D(0,1))?
\]
\item For a given $\rho _0 >1$ and a given $(d_i)_{i\ge 0}$, can we solve the interpolation problem \eqref{P32} with a loss $\rho \le \rho _0$? or without any loss?
\item  Can we extend Theorem \ref{thmintro} to the $N-$dimensional framework?
\end{enumerate}
As far as the complex variable approach is concerned, it would be very natural to replace in Laplace transform the integration over $\R _-$ by an integration over a finite interval. The main advantage would be to remove the assumption 
that the function $g$ in  we don't need to assume that the function  $g$ in \eqref{Y2} be analytic in a neighborhood of $\R^-$. 
Unfortunately, we can see that with this approach the loss cannot be less than $2$.  
For the sake of completeness, we give in appendix the proof of the following results concerning this approach. The first one asserts that the loss is at most $2^+$  for any function, and the second one that the loss is at least $2^+$ for certain functions. 
\begin{theorem}
\label{thm50}
Let $(a_n)_{n\ge 2}$ be a sequence of complex numbers satisfying
\be
\label{XYZ1}
|a_n|   \le  C\frac{n!}{R_0^n},\quad \forall n\ge 2
\ee
for some positive constants $C,R_0$.
Pick any $R\in (0,R_0)$ and let 
\be
G(x):=\int_0^R \phi (t) \exp (- t / x )  dt, \quad x\in (0,+\infty),  
\ee
where the function $\phi$ is defined by
\be
\phi (z):= \sum_{n\ge 2} a_n \frac{z^{n-1}}{ (n-1) ! }, \quad |z| <R_0.
\ee
Then $G^{(n)}(0^+)=a_n\, n!$ for all $n\ge 2$, and we have 
\be
|G^{ (n) } (x) | \le C'(n!)^2 (2/R)^n, \quad \forall x\in (0,+\infty ) 
\ee
for some constant $C'=C'(R )>0$
\end{theorem}
\begin{theorem}
\label{thm50a}
Let $R_0>R>0$ and pick any $p\in \N \setminus \{ 0, 1\} $ and any $C\in \R$. Let $(a_n)_{n\ge 2}$ be defined by
\[
a_n =\left\{
\begin{array}{ll}
C \displaystyle \frac{ p ! }{ R^p } & \textrm{ if }n=p, \\[3mm]
0&\textrm{ if } n\ne p. 
\end{array} 
\right.
\] 
Let $G$ and $\phi$ be as in Theorem \ref{thm50}. Then there does not exist a pair $(\hat C,\hat R)$ with $\hat C>0$, 
$\hat R>R$, and 
\be
|G^{ (n) } (x) | \le \hat C (n! )^2 (2/ \hat R)^n, \quad \forall x\in (0,R). 
\ee
\end{theorem}
Let us conclude this paper with some remarks. A step further would be the derivation of some exact controllability result with a continuous selection of the control in 
appropriate spaces of functions. This would be useful to derive local exact controllability results for semilinear parabolic equations in the same way as it was done 
before for the semilinear wave equation, NLS, KdV, etc.   To date, only the controllability to the trajectories was obtained for semilinear parabolic equations. The corresponding terminal states
are very regular.  Indeed, as it was noticed in \cite{AEWZ,MRRautomatica,meurer}, the solution of \eqref{LL1}-\eqref{LL4} with $h_0=h_1\equiv 0$ but $\psi (.,0)=\psi _0\in L^2(0,1)$ is such that 
\[
\psi (.,t)\in G^{\frac{1}{2}} ([0,1]),\quad 0<t\le T.
\]
In particular, $\psi (.,T)$ is an entire (analytic) function (that is, $\psi (.,T)\in H(\C )$). More precisely, the link between the Gevrey regularity and the order of growth of the entire function is revealed in the
\begin{proposition}
\label{prop50}
Let $T>0$ and $f\in G^\sigma ( [0,T] )$, $\sigma \ge 0$, and set 
\begin{eqnarray*}
g &:=& \inf \{  s\ge 0;\ f\in G^s ([0,T]) \} ,\\
\rho &:=& \inf \{ k>0,\ \exists r_0>0,\ \forall r>r_0, \ \max _{ |z|=r } | f (z) | <\exp (r^k) \} .    
\end{eqnarray*} 
Assume $g<1$. Then $f$ is an entire function of order $\rho \le (1-g)^{-1}$. If, in addition, $\rho \ge 1$, then $\rho = (1-g)^{-1}$.  
\end{proposition} 
For instance,  a function which is Gevrey of order $1/2$ (and not Gevrey of order less than $1/2$) is an entire function whose order is $\rho =2$. 
Thus, when dealing with reachable states, a gap in the Gevrey regularity (between $1/2$ and $1$) results in a gap in the order of growth of the entire function (between $2$ and $\infty$).
It follows that a local exact controllability for a semilinear heat equation in a space of analytic functions (if available) would dramatically improve the existing results, 
giving so far only the controllability to the trajectories, as far as the regularity of the reachable terminal states is concerned. 

Finally, the exact controllability to large constant steady  states of the dissipative Burgers equation
 \begin{eqnarray}
 \zeta _t-\zeta _{xx} +\zeta \zeta _x=0, && \quad x\in (-1,1), \ t\in (0,T), \label{LL5}\\
 \zeta (-1,t)=h_0(t),&&\quad t\in (0,T),  \label{LL6}\\
 \zeta (1,t)=h_1(t),&&\quad t\in (0,T), \label{LL7}\\
 \zeta (x,0)=\zeta _0(x),&& \quad x\in (-1,1), \label{LL8} 
 \end{eqnarray}
was derived in \cite{coron} from the null controllability of the system \eqref{LL1}-\eqref{LL4} (with $0\ne \psi _0\in L^2(-1,1)$) and Hopf transform $\zeta =-2\psi _x/\psi $.
It would be interesting to see whether Hopf transform could be used to derive an exact controllability result for \eqref{LL5}-\eqref{LL8}. 
Another potential application of the flatness approach (which yields {\em explicit} control inputs) is the investigation of the cost of the control (see \cite{lissy}).  
 \section*{Appendix}    
 \subsection{Proof of Theorem \ref{thm50}.}
 Pick any $n\geq 2$ and set $G_n(x) := G(x) - \sum_{p=2}^{n} a_p x^p$. It is clear that 
\[ G^{(n)}(x)=  a_n n! + G_n^{(n)}(x) \quad \forall x>0.\]
Since $x^p = \frac{1}{(p-1)!} \int_{0}^{+\infty} t^{p-1} \exp(-t/x) dt$, we have, with the series defining $\phi$,
$$
G_n(x)=
\underbrace{ \int_0^R \phi_n(t) \exp(-t/x)~ dt }_{= L_n(x)}-\underbrace{\int_R^{+\infty}\xi_n(t) \exp(-t/x)~ dt }_{= K_n(x)}
$$
where $\phi_n(t) := \sum_{p\geq n+1}  \frac{a_p t^{p-1}}{(p-1)!}$ and  $\xi_n(t):= \sum_{2 \leq p\leq n}  \frac{a_p t^{p-1}}{(p-1)!}$.
Thus
\begin{equation}\label{eq:Gn}
G^{(n)}(x)=a_n n! + L^{(n)}_n(x) - K^{(n)}_n(x).
\end{equation}
By Cauchy formula, we have 
 \begin{align*}
 L^{(n)}_n(x) &= \int_0^R \phi _n(t) \partial _x^n  [ \exp (-t / x) ]  dt = \frac{n!}{2 i\pi } \int_0^R \int_\Gamma \frac{\phi_n(t) \exp(-t/z) }{(z-x)^{n+1}}~dzdt, \\
 K^{(n)}_n(x) &= \int_R^\infty  \xi _n(t) \partial _x^n [ \exp (-t / x) ]  dt =  \frac{n!}{2 i\pi } \int_R^\infty  \int_\Gamma \frac{\xi _n(t) \exp(-t/z) }{(z-x)^{n+1}}~dzdt, 
\end{align*}
where $\Gamma$ is any (smooth) closed path  around $z=x$ and not around the essential singularity $0$.
Consider the following family of circles centered at $x$ and with radius  $(1-\epsilon)x$ (with $0<\epsilon <1$):
 $$
 \Gamma_\epsilon := \left\{ x+x(1-\epsilon) e^{i\theta}~|~\theta\in[-\pi,\pi] \right\}.
 $$
Letting $\epsilon \searrow 0$, we infer from Lebesgue dominated convergence theorem that  
\begin{align*}
  L^{(n)}_n(x)& = \tfrac{n!  }{2\pi  x^n } \int_{0}^{R} \int_{-\pi}^{\pi} \phi_n(t)  \exp\left(-\frac{t}{2 x}\right) \exp\left(i \left( \frac{t }{2x}\tan(\theta/2) - n \theta\right)\right) d\theta ~dt,\\
  K^{(n)}_n(x)& = \tfrac{n!  }{2\pi  x^n } \int_{R}^{+\infty} \int_{-\pi}^{\pi} \xi_n(t)  \exp\left(-\frac{t}{2 x}\right) \exp\left(i \left( \frac{t }{2x}\tan(\theta/2) - n \theta\right)\right) d\theta ~dt.
\end{align*}

Since $|\xi_n(t)| \leq  C \sum_{p=2}^{n} (p / R_0) (t/ R_0)^{p-1}$, we have for any $t\geq R$
$$
|\xi_n(t)|\leq (C/ R_0)  (t/R)^{n-1} \sum_{p=2}^{n} p (R/ R_0)^{p-1} \leq   (C/R_0)  (t/R)^{n-1} (1-R/R_0)^{-2}.
$$
Consequently
$$
 |K^{(n)}_n(x)| \leq C \tfrac{n! R R_0 }{ (R_0 - R)^2 x^n  } \int_{R}^{+\infty}   (t/R)^{n-1} \exp\left(-\frac{t}{2 x}\right)~dt/R.
$$
With $ \int_{R}^{+\infty}   (t/R)^{n-1} \exp\left(-\frac{t}{2 x}\right)~dt/R\leq (n-1)! (2x/R)^{n}$, we obtain
\begin{equation}\label{eq:Kn}
 |K^{(n)}_n(x)| \leq  C \tfrac{R R_0 }{ (R_0-R)^2 } n! (n-1)!  (2/R)^{n} .
\end{equation}

On the other hand, we have
\begin{eqnarray*}
| L^{(n)}_n(x)|  &\leq& 
 \displaystyle \frac{n!  }{x^n } \int_{0}^{R} \sum_{p \geq n+1} |a_p| \frac{ t^{p-1}}{(p-1)!}  \exp\left(-\frac{t}{2 x}\right)   ~dt \\
&\leq& \displaystyle \frac{C n!  }{x^n } \int_{0}^{R} \sum_{p \geq n+1} p  \frac{ t^{p-1}}{R_0^p}  \exp\left(-\frac{t }{2 x}\right)   ~dt \\
&=& \displaystyle \frac{C R n!  }{R_0 } \sum_{p \geq n+1} \int_{0}^{1} p \tau ^{p-1} (R/R_0)^{p-1}  x^{-n}\exp\left(-\frac{R\tau}{2 x}\right)   ~d\tau .
\end{eqnarray*}
Since $ x^{-n}\exp\left(-\frac{R\tau}{2 x}\right)  \leq C\sqrt{n}  \left(\frac{2n}{eR\tau}\right)^n$, we obtain
$$
| L^{(n)}_n(x)| \leq \tfrac{C R n!  }{R_0 } \left(\frac{2n}{eR}\right)^n \sum_{p \geq n+1} \int_{0}^{1} p \tau ^{p-1-n} (R/R_0)^{p-1} ~d\tau .
$$
Combined with 
\[ 
\sum_{p \geq n+1} \int_{0}^{1} p \tau ^{p-1-n} (R/R_0)^{p-1} ~d\tau = \sum_{p\geq n+1} \frac{p}{p-n} (R/R_0)^{p-1}\leq \frac{(n+1) (R/R_0)^{n}}{1-R/R _0}, \] 
this yields 
\begin{equation}\label{eq:Ln}
 | L^{(n)}_n(x)| \leq \tfrac{C R  }{R_0 - R } n! (n+1)\sqrt{n}  \left(\frac{n}{e}\right)^n  \left(\frac{2}{R_0}\right)^n.
\end{equation}

From~\eqref{eq:Gn},~\eqref{eq:Kn} and~\eqref{eq:Ln} and Stirling formula, we conclude that there exists some constant $C' >0$ such that
$$
|G^{(n)}(x)| \leq C' (n!)^2 (2/R)^n, \quad \forall x\in (0,+\infty). 
$$

\subsection{Proof of Theorem \ref{thm50a}.} 
With this choice of the sequence $(a_n)_{n\ge 2}$, we have (after some integrations by parts)
\[
G(x)= a_p    \int_0^R \frac{t^{p-1}}{(p-1)!}  \exp (-t/x)dt = a_px^p + P(x)e^{-R/x}
\]
where $P(x)$ is a polynomial function with real coefficients and of degree $p$. For $n>p$, the  derivative of order $n$ of $G(x)$ 
and  of $P(x)e^{-R/x}$ coincide. Letting $g(x)=P(Rx)e^{-1/x}$, we have that $P(x)e^{-R/x}=g(x/R)$ and hence for $n>p$
\[
G^{ (n) } (x)  = R ^{-n} g^{(n)} (x/R).
\]
To conclude, we need the following lemma, which is of interest in itself. 

\begin{lemma}
\label{lem-Pierre-2}
\label{lem:Pexp}
 Let  $F\neq 0$  be an holomorphic function without singularity in the closed disk of radius $x^*>0$ centered at $0$. 
 Assume also that $F$  take real values on the real segment $[-x^*,x^*]$. Consider the $C^\infty$  function 
 $g :x\in [0,x^* ] \mapsto F(x)e^{-1/x}\in \R$. Then there exists a number $C>0$ such that  
 \begin{equation}
 \label{PL1}
  |g^{(n)}(x)| \leq  C  \left(\frac{1}{2}\right)^n (2n)! , \quad \forall n\in \N , \ \forall x\in [0,x^*].
\end{equation}  
   Moreover, there does not exist a pair $(C',R')$ with $C'>0$ and  $R' \in (0,\frac{1}{2})$ such that  
  \begin{equation}
  \label{PL2}
  |g^{(n)}(x)| \leq  C' \left(R'\right)^n (2n)! ,  \quad \forall n\in \N , \ \forall x\in [0,x^*].
  \end{equation}
\end{lemma}
\noindent
{\em Proof of Lemma \ref{lem-Pierre-2}.} 
Since $g$ is an holomorphic function without singularity in a neighborhood of the real segment $[x^*/2,x^*]$, we infer from Cauchy
formula that for some constants $K,r>0$ we have 
\[
|g^{n)}(x)| \le \frac{K}{r^n} n!, \quad \forall n\in \N, \ \forall x\in [x^*/2, x^*]. 
\]
Therefore, it is sufficient  to prove the lemma for $x\in(0,x^*/2]$. (Note that all the derivatives of $g$ vanish at $x=0$.)
  
 Take $x\in(0,x^*/2]$. By the Cauchy formula
 $$
 g^{(n)}(x)= \frac{n!}{2 i\pi } \int_\Gamma \frac{g(z)}{(z-x)^{n+1}}~dz
$$
where $\Gamma$ is a closed path  around $z=x$, but not around the essential singularity $0$, and inside the disk of radius $x^*$ and centered at $0$. 
Consider the following family of circles centered at $x$ and  of radius  $(1-\epsilon)x$ with  $\epsilon$ tending to $0^+$:
 $$
 \Gamma_\epsilon= \left\{ x+x(1-\epsilon) e^{i\theta}~|~\theta\in[-\pi,\pi] \right\}.
 $$
 We have that
 \begin{equation} \label{eq:gn}
   g^{(n)}(x)= \frac{n!}{2\pi (1-\epsilon)^n x^n } \int_{-\pi}^{\pi} F\left( x+x(1-\epsilon) e^{i\theta}\right) e^{\frac{-1}{x(1+(1-\epsilon)e^{i\theta})}}e^{-i n \theta} d\theta
 .
 \end{equation}
Since
$$
\frac{-1}{x(1+(1-\epsilon)e^{i\theta})} =
 \frac{-1}{x} ~\frac{1+(1-\epsilon)\cos\theta - i (1-\epsilon)\sin\theta)}{1+(1-\epsilon)^2+2(1-\epsilon)\cos\theta}
$$
we have, for each $\theta\in (-\pi,\pi  )$ and $x\in (0,x^*/2) $,
$$
\lim_{\epsilon \mapsto 0^+} \frac{-1}{x(1+(1-\epsilon)e^{i\theta})}
=  \frac{-1}{2x} + i \frac{\sin\theta}{2x(1+\cos\theta)}=  \frac{1}{2x} (-1 + i \tan(\theta/2)) 
.
$$
Moreover
 $$
\forall \theta \in[-\pi,\pi], \forall \epsilon\in[0,1/2],  ~\forall x\in[0,x^*/2],~ \left| F\left( x+x(1-\epsilon) e^{i\theta}\right) e^{\frac{-1}{x(1+(1-\epsilon)e^{i\theta})}} \right|\leq \| F \| _\infty
$$
where $ \| F\| _\infty = \sup _{|z|\le x^* } | F (z)|$.  
Using Lebesgue's dominated convergence theorem, we can take the limit as $\epsilon$ tends to $0$ in~\eqref{eq:gn} to get
\begin{equation} \label{eq:gn0}
   g^{(n)}(x)= \frac{n! \exp\left(-\frac{1}{2 x}\right) }{2\pi  x^n } \int_{-\pi}^{\pi} F \left( x\left(1+ e^{i\theta}\right)\right)\exp\left(i \left( \frac{\tan(\theta/2)}{2x} - n \theta\right)\right) d\theta
 .
 \end{equation}
Consequently
$$
  |g^{(n)}(x)|\leq \| F  \| _\infty \frac{n! \exp\left(-\frac{1}{2 x}\right) }{x^n }
$$
But the maximum of $x\mapsto x^{-n} \exp\left(-\frac{1}{2 x}\right)$  for $x>0$ is reached at $x= 1/(2n)$, and thus
$$
\max_{x\in[0,x^*]}   |g^{(n)}(x)| \leq \|  F \| _\infty  n! \left(\frac{2n}{e}\right)^n.
$$
Since $n! \sim \sqrt{2\pi n} (n/e)^n$, we have  that $(2n)!\sim \sqrt{4\pi n} (2n/e)^{2n}$, and hence 
\[
n! \left(\frac{2n}{e}\right)^n \sim \left(\frac{1}{2}\right)^{n+1/2} (2n)! . 
\]
This gives \eqref{PL1} with a constant $C>0$ that depends linearly on $\| F \| _\infty$.

For $x=1/(2n)\leq x^*/2$, we have
$$
 g^{(n)}\left(\frac{1}{2n}\right)= \frac{n! \left(\frac{2n}{e}\right)^n }{2\pi  } \int_{-\pi}^{\pi}   F\left( \frac{1+ e^{i\theta}}{2n}\right)\exp\left(i n \left( \tan(\theta/2)-\theta\right)\right) d\theta
 .
$$
For $n$ large we can evaluate  this oscillatory integral by using the stationary phase method.
Since $F$ takes real values on $[-x^*, x^*]$, we have that $F^{(n)}(0)\in \R$ for all $n\in \N$ and hence that
$\overline{F(z)}=F(\overline{z})$ for $|z|<x^*$. It follows that 
$$
 \int_{-\pi}^{\pi}  F\left( \tfrac{1+ e^{i\theta}}{2n}\right) \exp\left(i n \left( \tan(\theta/2)-\theta\right)\right) d\theta
 = 2 \Re\left\{  \int_{0}^{\pi}   F\left( \tfrac{1+ e^{i\theta}}{2n}\right)\exp\left(i n \left( \tan(\theta/2)-\theta\right)\right) d\theta\right\}
 .
$$
Since $F$ is holomorphic in a neighborhood of $0$, there exist an integer  $r\geq 0$ and a holomorphic  function $Q$  such that  
 $F\left( \tfrac{1+ e^{i\theta}}{2n}\right)= \left( \tfrac{1+ e^{i\theta}}{2n}\right)^r Q\left( \tfrac{1+ e^{i\theta}}{2n}\right)$ 
where $Q(0)\neq 0$. (Note that $Q(0)\in \R$.)

As far as the integral $\int_{0}^{\pi}  F\left( \tfrac{1+ e^{i\theta}}{2n}\right) \exp\left(i n \left( \tan(\theta/2)-\theta\right)\right) d\theta$
in concerned, we readily see that the phase $\phi(\theta)=\tan(\theta/2)-\theta$  is stationary at only one point on $[0, \pi]$, namely
$\bar\theta=\pi/2$. We have  
 $\phi'(\bar\theta)=0$ and $\phi''(\bar\theta)=1>0$. Thus we can write the following asymptotic approximation
\begin{multline*}
  \int_0^\pi  F\left( \frac{1+ e^{i\theta}}{2n}\right) \exp(in \phi(\theta)) d\theta 
  = Q(0) \left( \frac{1+ i}{2n}\right)^r  \left( \sqrt{\frac{2\pi}{n \phi''(\bar\theta)} }  \exp\left( i ( n \phi(\bar\theta) + \pi/4) \right)  + o(1/\sqrt{n}) \right) \\
=  \frac{Q(0)}{2^{r/2} n^r}   \left(\sqrt{\frac{2\pi}{n } } \exp\left( i ( n (1-\pi/2) + (r+1)\pi/4)\right) + o(1/\sqrt{n})\right)
\end{multline*}
 for $n$ large,  see e.g. \cite[page 279]{BenderOrszagBook} or \cite{hormander1}. Then there exist two functions of $n$ vanishing at infinity, $\eta$ and $\mu$, such that
 $$
  g^{(n)}\left(\frac{1}{2n}\right) =     2  \frac{Q(0)}{2^{r/2} n^r} \sqrt{\frac{\pi}{n } }\big(1+\eta(n)\big) \left(\frac{1}{2}\right)^{n} (2n)!   \big( \cos\left(  n (1-\pi/2) + (r+1)\pi/4 \right)  + \mu(n)\big)
  .
 $$
 If we could find  $R' \in (0,\frac{1}{2})$ and $C'>0$ such that  for all  $n\in\N$ 
 and all $x\in[0,x^*/2] $, $|g^{(n)}(x)| \leq  C  \left(R'\right)^n (2n)!  $,
then  we would have
 $$
\left| 2  \sqrt{\pi } \Big(1+\eta(n)\Big)  \bigg( \cos\left(  n (1-\frac{\pi}{2}) + \frac{(1+r)\pi}{4}\right)  + \mu(n)\bigg) \right|
 \leq  \frac{C'2^{r/2}}{Q(0)}n^{r+1/2} (2R')^n
 .
 $$
 Since $\lim_{n\mapsto +\infty} n^{r+1/2} (2R')^n =0$, we would obtain a  contradiction to  the fact that the set of limit points of the sequence $\cos\left(  n (1-\frac{\pi}{2}) + \frac{(1+r)\pi}{4}\right)$ is $[-1,1]$.
This completes the proof of Lemma \ref{lem-Pierre-2}. \qed

Let us go back to the proof of Theorem \ref{thm50a}. Apply Lemma \ref{lem-Pierre-2} with $x^*=1$ and $F(x)=P(Rx)$. 
If there exists a pair $(\hat C,\hat R)$ with $\hat C>0$, $\hat R>R$ and such that
\[
|G^{ (n) } (x) | \le \hat C (n! )^2 (2/ \hat R)^n, \quad \forall x\in (0,R),
\]
then, setting $\hat R=\rho R$ with $\rho >1$,  we obtain that for all $x\in (0,1)=(0,x^*)$ and all $n\in \N$ 
\begin{eqnarray*}
| g^{ (n) } (x) | &=& R^n | G^{(n)} (Rx)| \\
&\le &  \hat C (n!)^2  \left(  \frac{2}{\rho} \right) ^n  \\
&\le& Const. \frac{\sqrt{n}  (2n)!}{2^{2n}}    \left(  \frac{2}{\rho} \right) ^n\\
&\le&  Const. \frac{\sqrt{n}  (2n)!}{ (2\rho )^n} \\ 
&\le&C' (R')^n (2n)!
\end{eqnarray*}
for some  $C',R'$ with $(2\rho )^{-1}<R'<2^{-1}$ and $C'>0$. But this is not possible, according to Lemma \ref{lem-Pierre-2}. The proof of Theorem  \ref{thm50a}
is achieved. \qed 

 \subsection{Proof of Proposition \ref{prop50}.}    
 We introduce some notations borrowed from  \cite{levin-book}. If an inequality $f(r)<g(r)$ holds for sufficiently large  values of $r$ (i.e. for all $r>r_0$ for some $r_0\in \R$), we shall write
 $f(r)\asy g(r)$ (``as'' for {\em asymptotic}).  
 
 Let $f(z)=\sum_{n=0}^\infty c_nz^n$ be an entire function ($f\in H(\C )$). Let $M_f(r):=\max_{|z|=r} |f(z)|$. Then the order (of growth) of the entire function $f$ is 
 \[
 \rho = \inf \{ k > 0; \ M_f(r) \asy \exp (r^k) \} \in [0,\infty ].  
 \]   
 The following results will be used thereafter.
 \begin{lemma} \cite[Lemma 1 p. 5]{levin-book}.
 \label{lem51}
 If the asymptotic inequality
 \begin{equation}
 \label{ABC1}
 M_f(r) \asy e^{Ar^k}
 \end{equation}
 is satisfied, then 
 \begin{equation}
 \label{ABC2}
 |c_n| \asy \left( \frac{eAk}{n}\right) ^\frac{n}{k} 
 \end{equation}
 \end{lemma} 
 
 \begin{lemma} \cite[Lemma 2 p. 5]{levin-book}.
 \label{lem52}
If the asymptotic inequality \eqref{ABC2} is satisfied, then 
\begin{equation}
\label{ABC3}
 M_f(r) \asy e^{(A+\varepsilon ) r^k}\quad  \forall \varepsilon >0. 
\end{equation}
\end{lemma} 
 
Let $f$ be as in the statement of Proposition \ref{prop50}. Then for all $s\in (g,1)$, there are some constants $C=C(s)>0$ and $R=R(s)>0$ such that 
\begin{equation}
\label{ABC11}
|f^{(n)} (t) | \le C \frac{(n!)^s}{R^n} \quad \forall t\in [0,T].
\end{equation}
Let $c_n := f ^{ (n) } (0) / n! $ for all $n\in \N$. Then the series $\sum_{n=0}^\infty c_nz^n$ converges for all $z\in \C$ (since $s<1$), and we have for all $z\in [0,1]$
\begin{equation}
\label{ABC12} 
f(z) = \sum_{n=0}^\infty c_nz^n
\end{equation} 
(see \cite[19.9]{rudin}).
Thus $f$ can be extended as an entire function by using \eqref{ABC12} for $z\in \C$.  Set 
\[
k := (1-s)^{-1}.
\]
It follows from \eqref{ABC11} and Stirling formula that
\[
|c_n| \le  C [(n\, !)^{1-s} R^n ]^{-1}\le C' \left( \frac{e}{n} \right) ^{\frac{n}{k}}  [n^\frac{1-s}{2} R^n ]^{-1} \asy \left( \frac{eAk}{n} \right) ^\frac{n}{k}  
\]  
for some positive constants $C,C'$, and $A$.  We infer from Lemma \ref{lem52} that 
\[
M_f (r) \asy e^{ (A+\varepsilon ) r^k} \asy e^{r^{k+\varepsilon}} \quad \forall \varepsilon >0,
\]
and hence $\rho \le k = (1-s)^{-1}$. Letting $s\searrow g$, we obtain 
\be
\label{ABC30}
\rho \le (1-g)^{-1}.
\ee

Assume in addition that $\rho \ge 1$. 
We infer from the definition of $\rho $ that 
\[
M_f(r) \asy e^{r^{\rho + \varepsilon}}\quad \forall \varepsilon >0. 
\] 
Pick any $\varepsilon >0$ and let $k:=\rho + \varepsilon  >1$. It follows from Lemma \ref{lem51} that 
\[
|c_n| \asy \left(  \frac{eAk}{n} \right) ^\frac{n}{k}  
\]
and hence 
\[
|c_n| \le C \frac{(n!)^{-1/k }}{R^n}\quad \forall n\in \N
\] 
for some positive constants $C$ and $R$. It follows that for all $t\in [0,T]$ and $q\in \N$ 
\begin{eqnarray*}
| f^{ (q) } (t) | 
&=& \vert \sum_{n\ge q} \frac{n!}{ (n-q)! } c_n z^{n-q}   \vert  \\
&\le& C \sum _{p\ge 0} \frac{ (p+q)!^{1-k^{-1}} }{p !} \frac{|z|^p}{R^{p+q}}  \\
&\le& C 2^{ks} q! ^s \sum_{p\ge 0} \frac{(2^s |z| )^p }{ p! ^{k^{-1}} }  
\end{eqnarray*}
where we have set $s:=1-k^{-1}$ and
used the inequality $(p+q)! \le 2^{p+q} p! q!$.   Since $k>0$, the series in the last inequality is convergent, and we infer that $f \in G^s([0,T])$. Thus 
$g\le s=1 - (\rho + \varepsilon )^{-1}$. Letting $\varepsilon \searrow 0$, we obtain that $g\le 1-\rho ^{-1}$. Combined with \eqref{ABC30}, this yields $\rho = (1-g)^{-1}$.  \qed 
\begin{remark}
Note that we can have $0=\rho <(1-g)^{-1}=1$ (pick e.g. any polynomial function $f\in \C [z]$).  
\end{remark}


\end{document}